\documentclass[preprint,12pt]{elsarticle}
\usepackage{tikz}
\usepackage{pgfplots}
\usepackage{graphicx} \usepackage{amsthm,amsfonts,amssymb,amscd,amsmath,enumerate,verbatim,calc,graphicx,geometry}
\usepackage[all]{xy}
\usepackage{tikz}
\usetikzlibrary{decorations.fractals}
\usepackage{pgfplots}
\usepackage{hyperref}
\usepackage{amsmath}
\newtheorem{theorem}{Theorem}[section]
\newtheorem{lemma}[theorem]{Lemma}
\newtheorem{proposition}[theorem]{Proposition}
\newtheorem{corollary}[theorem]{Corollary}
\theoremstyle{definition}
\theoremstyle{definitions}
\newtheorem{definition}[theorem]{Definition}

\newtheorem{example}[theorem]{Example}
\theoremstyle{notations}

\theoremstyle{remarks}

\journal{ }

\begin{document}

\begin{frontmatter}



\title{A Discrete Topological Complexity of Discrete Motion Planning }

 
\author[]{Hadi~Hassanzada}
\ead{hadihz1994@gmail.com}
\author[]{Hamid~Torabi\corref{cor1}}
\ead{h.torabi@um.ac.ir}
\author[]{Hanieh~Mirebrahimi}
\ead{h$_{-}$mirebrahimi@um.ac.ir}
\author[]{Ameneh~Babaee}
\ead{am.babaee@mail.um.ac.ir}
\address{Department of Pure Mathematics, Ferdowsi
University of Mashhad, P.O.Box 1159-91775, Mashhad, Iran}

\author[]{~\corref{cor1}}
\ead{email}

\address{}
\cortext[cor1]{Corresponding author}

\begin{abstract}
In this paper we generalize the discrete $r$-homotopy to the discrete $(s,r)$-homotopy. Then by this notion, we introduce the discrete motion planning for robots which can move discreetly. Moreover, in this case the number of motion planning, called discrete topological complexity, required for these robots is reduced. Then we prove some properties of discrete topological complexity; For instance, we show that a discrete motion planning in a metric space $X$ exists if and only if $X$ is a discrete contractible space. Also, we prove that the discrete topological complexity depends only on the strictly discrete homotopy type of spaces.
\end{abstract}


\end{frontmatter}

\section{Introduction}
The motion planning of robots depends on geometric, physical and temporal constraints \citep{6}. Topological complexity of motion planning was introduced by Farber \cite{1} to investigate geometric constraints of robotic motion. We rcall from \cite{1} that the topological complexity of the motion planning is equal to the minimum number of continuous motion planning algorithms in any configuration space. Topological complexity of a path connected space $X$, denoted by $TC(X)$, is the smallest number $k$ such that the Cartesian product $X\times X$ may be covered by k open subsets $U_i$ and there are continuous motion plannings $s_i:U_i \to PX$ where $PX$ denotes the space of all continuous paths in $X$. Farber proved that a continuous motion planning $s:X \times X \to PX$ exists if and only if $X$ is a contractible space. The concept of TC is related to the  Lusternik-Schnirelman category of the topological spaces. LS category of a space is a number that counts the distance of space with contractibility. If this number is larger, the space is further away from contractibility.  LS category of a space $X$, denoted by $cat(X)$, is the smallest number $k$ such that $X$ covered by $k$ open subsets $V_i$ where inclusion map $i:V_i\to X$ is null homotopic. Farber proved that the amount of topological complexity of $X$ is between $cat(X)$ and $cat(X\times X)$. Moreover, he provided an upper bound by the topological dimension and a lower bound by the cohomology theory for topological complexity of paracompact spaces. In this paper, to reduce the amount of topological complexity for robots which can move discreetly we use the discrete motion planning instead of continuous motion planning.

Barcelo, Capraro and White introduced the discrete homotopy theory for metric spaces in \cite{2} (see also \cite{4}). Contractibility of (contractible) space induces a motion planning on $X$. By using the concept of discrete homotopy theory, it may not to define discrete contractible space for any positive scale $r$. In this paper we introduce discrete $r$-contractible space by generalizing the difinition of $r$-homotopy to $(s,r)$-homotopy to induce discrete motion planning.

 In section 1, we prove that every discrete contractible space is $r$-connected and we show that the discrete fundamental group of a discrete contractible space vanishes. In the following, we prove that a discrete motion planning on a metric space $X$ exists if and only if $X$ is a discrete contractible space. In section 2, we show that discrete TC depends on the discrete homotopy type. In section 3, we obtain some lower and upper bound for this concept by difinning discrete LS category. By studying the properties of discrete topological complexity, we also show that the discrete topological complexity is may less than the Farber's topological complexity. For the Hawaiian earring, we show that the discrete topological complexity is finite, while the topological complexites of the Hawaiian earring or other examples such as Koch's snowflake or Sierpiński's carpet are infinite.

\section{Discrete Topological Complexity }\label{sec1}
Let $(X,d_X)$ and $(Y,d_Y)$  be metric spaces and $r>0$. A function $ f:X \to Y$ is an r-Lipschitz map if $d_Y(f(x_1),f(x_2))\leq rd_X(x_1,x_2) $, for all $x_1,x_2 \in X$. The set $[m]:=\{0,\ldots m\}$ is equipped with the metric $d(a,b)=|a-b|$ and the cartesian product $X\times [m]$ is equipped with the $\ell^1$-metric. The following definition were introduced and studied in \cite{2}. The map $\gamma :[m]\to X$ is an $r$-path if $d_X (\gamma (i) , \gamma (i+1))\leqslant r$, for all $i\in [m]$. The metric space $X$ is $r$-conected if for $x,y\in X$, there is an $r$-path from $x$ to $y$. Also let $X, Y$ be metric spaces, and let $f, g : X \to Y$ be $r$-Lipschitz maps. Then $f$ and $g$ are $r$-discrete homotopic if there exists a non-negative integer $m$ and an $r$-Lipschitz map $F : X \times \{0, \ldots, m\} \to Y$ such that $F(-, 0) = f$ and $F(-, m) = g$. Also the map $F:X\times [m]\to Y$ is an $r$-homotopy map if and only if $F(-,i)$ is an $r$-Lipschitz map for every $i\in [m]$ and $F(x,-)$ is an $r$-Lipschitz map for every $x\in X$. Now we generalize this definition as follows.
\begin{definition}\label{(s,r)-homotopy}
Let $(X,d_X)$ and $(Y,d_Y)$  be metric spaces and $f,g:X\to Y$ be $s$-Lipschitz maps. Then $f$ and $g$ are $(s,r)$-homotopic if there exsists a non-negetive integer $m$ and a map $F:X\times [m]\to Y$ such that for all $x\in X$ we have $F(x,0)=f(x)$ and $F(x,m)=g(x)$ and also, for every $i\in [m]$, $F(-,i)$ is an $s$-Lipschitz map and for every $x\in X$, $F(x,-)$ is an $r$-Lipschitz map. We denote it by $F:f\simeq_{(s,r)} g$.
\end{definition}
One can verify that the $(s,r)$-homotopy  is an equivalence relation and if $s=r$ then $(s,r)$-homotopy is equivalent to $r$-discrete homotopy in the sense of  \cite{2}. Recall that a continuous motion planning is a continuous map $s:X\times X\to PX$ with $ \pi \circ s=id$ where $\pi :PX\to X\times X$ mapping every path to the pair of initial and terminal points.  Farber proved that a continuous motion planning function $s:X \times X \to PX$ exists if and only if $X$ is contractible. In the following, we present this fact for discrete motion planning by defining $r$-contractibility.
\begin{definition}\label{r-null-homotopic}\label{r-contractible}
Let $f$ be an $s$-Lipschitz map. The map $f:X\to Y$ is called $r$-null-homotopic if there is an $(s,r)$-homotopy map between $f$ and a constant map. Inparticular, a metric space $X$ is an $r$-contractible space if the $1$-Lipschitz map identity on $X$ is an $r$-null-homotopic, i.e. there is a $(1,r)$-homotopy map $F:X\times [m] \to X$ between the identity map and a constant map.
\end{definition}

In the following example we show that  the $n$-dimensional closed ball $B^n$ is an $r$-contractible space for every $r>0$.
\begin{example}
Let $B^n$ be the closed ball in $\mathbb{R}^n$ of radius $m^\prime$ centered at the origin and let $r>0$. There is an integer $m$ such that $\dfrac{m^\prime}{m}<r$. Define a map $F:‌B^n\times [m] \to B^n$ such that $F(x,i)=x-\dfrac{ix}{m}$. We show that $F$ is a $(1,r)$-homotopy between $id_{B^n}$ and the constant map $c_0=0$. For every $x\in B^n$ and $i,j \in [m]$, we have
$d(F(x,i),F(x,j))=|(x-\dfrac{ix}{m})-(x-\dfrac{jx}{m})=|\dfrac{x}{m}(j-i)|\leq \dfrac{m^\prime}{m}|j-i|\leq rd(i,j)$.
  So $F(x,-)$ is an $r$-Lipschitz map. For every $i\in [m]$ and $x,y\in B^n$, we have
$$d(F(x,i), F(y,i))=d( x-\dfrac{ix}{m},y-\dfrac{iy}{m})=|x-\dfrac{ix}{m}-y+\dfrac{iy}{m}|$$
$$=(1-\frac{i}{m})|x-y|=(1-\frac{i}{m})d(x,y)\leq d(x,y)$$

So $F(-,i)$  is a $1$-Lipschitz map. Also $F(x,0)=x$ and $F(x,m)=0$. Thus the closed ball $B^n$ is an $r$-contractible space. 
\end{example}

\begin{proposition} 
Every $r$-contractible space is $r$-connected.
\end{proposition}
\begin{proof}
Let $X$ be an $r$-contractible. Hence there is an integer $m$ and a $(1,r)$-homotopy $ F:X\times [m]\to X$ between the identity map and a constant map. Let $x,y\in X$ and
\begin{displaymath}
\gamma (i)= \left\{
\begin{array}{lr}
F(x,i)                     &      0 \leq i\leq m \\
F(y,2m-i)     &     m \leq i \leq 2m
\end{array}
\right.
\end{displaymath}
Then $\gamma $ is an $r$-path from $x$ to $y$. 
\end{proof}

Recall from \cite{2} that an $r$-loop based at $p\in X$ is an $r$-path $\gamma:[m]\to X$ such that $ \gamma(0)=\gamma(1)=p$. Consider the set of $r$-loops based at $p$, denoted by $C_r(X,p)$, with the operation of concatenation. The $(r,r)$-homotopy is equivalent to the $r$-homotopy. Moreover, the $r$-homotopy is an equivalence relation on $C_r(X,p)$  and denoted by $\simeq_r$. The discrete fundamental group of metric space $X$ at scale $r$ was defined in \cite{2} as follows: 
$$A_{1,r}(X,p):=\dfrac{C_r(X,p)}{\simeq_r}$$
Now we show that the discrete fundamental group of discrete contractible space is the trivial group.
\begin{lemma}\label{lemma}
Let $X$ be a metric space. If $X$ is $r$-contractible, then $A_{1,r}(X,x_0)$ is trivial group for all $x_0 \in X$.
\end{lemma}
\begin{proof}
Let $x_0\in X$ and $F:X\times [m]\to X$ be a $(1,r)$-homotopy between the identity map and the constant map $c_{x_0}$ such that $F(x,0)=x_0$ and $F(x,m)=x$. Now let $f:[n] \to X$ be an $r$-loop based at $x_0$. Let $x_i=f(i)$ for all $0\leq i \leq m$. Consider the following matrix.
\begin{center}
$\begin{pmatrix}
 F(x_0,0) &F(x_1,0) & \cdots & F(x_n,0) \\

 F(x_0,1) & F(x_1,1)  & \cdots  & F(x_n,1)  \\

\vdots & \vdots & \vdots & \vdots   \\

 F(x_0,m-1) & F(x_1,m-1) &  \cdots  & F(x_n,m-1) \\

 F(x_0,m) & f(x_1,m)  & \cdots & F(x_n,m) \\
\end{pmatrix}
$
\end{center}
Every column is an $r$-path, since $F(x,-)$ is an $r$-Lipschitz map. Also every row is an $r$-path, since for $i\in [m]$ we have $F(-,i)$ is a $1$-Lipschitz map and $f$ is an $r$-path so $d(F(x_j,i),F(x_{j+1},i))\leq d(x_j,x_{j+1})\leq r$. But every row is not necessary a loop. Now consider the following matrix.
{\footnotesize
\[\begin{pmatrix}
F(x_0,0) & \cdots & F(x_0,0) & F(x_0,0) & \cdots & F(x_n,0) & F(x_0,0) & \cdots & F(x_0,0) \\

F(x_0,0) & \cdots & F(x_0,0) & F(x_0,1)  & \cdots  & F(x_n,1) & F(x_0,0)  & \cdots & F(x_0,0) \\

\vdots & \vdots & \vdots & \vdots & \vdots & \vdots & \vdots & \vdots & \vdots \\

F(x_0,0) & \cdots & F(x_0,m-2) & F(x_0,m-1) &  \cdots  & F(x_n,m-1) & F(x_0,m-2)  & \cdots & F(x_0,0) \\

F(x_0,0) & \cdots & F(x_0,m-1) & F(x_0,m)  & \cdots & F(x_n,m) & F(x_0,m-1)  & \cdots & F(x_0,0) \\
\end{pmatrix}
\]}

$F(x_0,-)$ is $r$-Lipschitz. Therefore, every column is an $r$-path. Also, every row is an $r$-loop based at $x_0$, since for $i\in [m]$ we have $F(-,i)$ is $1$-Lipschitz and note that $x_0=f(0)=f(n)=x_n$. Therefore, if $\gamma : [m] \to X$ is an $r$-path with $\gamma(i)=F(x_0,i)$, then according to the last matrix, $[\gamma \ast f \ast \gamma^{-1}]=1$.  Thus $A_{1,r}(X,x_0)$ is the trivial group.
\end{proof}
\begin{example}
The unit circle $S^1$ is $r$-contractible for $r\geq 2$ and it's not $r$-contractible for $r<2$. Let $r\geq 2$ then the map  $F:S^1\times [1] \to S^1$ such that $F(x,0)=x$ and $F(x,1)=(1,0)$  is a  $(1,r)$-homotopy between the identity map and the costant map $c_{(1,0)}$. Let $r< 2$ and $S^1$ be $r$-contractible. By Lemma \ref{lemma}, $A_{1,r}(S^1,x_0)$ is trivial group which is a contradiction by $A_{1,r}(S^1,x_0) \cong \mathbb{Z}$. Thus $S^1$ is not $r$-contractible for $r<2$.
\end{example}

For metric spaces we can define discrete motion planning as follows.
\begin{definition}\label{r-motion planning}
The discrete motion planning in $(X,d_X)$ consists of $1$-Lipschitz map $s:X\times X \to X^{[m]}$ such that $\pi \circ s=id_{X\times X}$, for some $m\in \mathbb{N} $, where $X^{[m]}$ is the set of $r$-paths $\gamma : [m] \to X$ equipped with the uniform metric and $\pi:X^{[m]} \to X\times X$ is the map with $\pi (\gamma ) = (\gamma (0) , \gamma (m))$. We call $s$ an $r$-motion planning for $X$.
\end{definition}
In continuous motion planning, the robot must go around through small and large holes to reach a specific goal. But in discrete motion planning, the holes that are smaller than the robot's step length can be easily passed. see figure \ref{Robot}.
\begin{figure}[htbp]
\begin{center}
\includegraphics[width =8cm ]{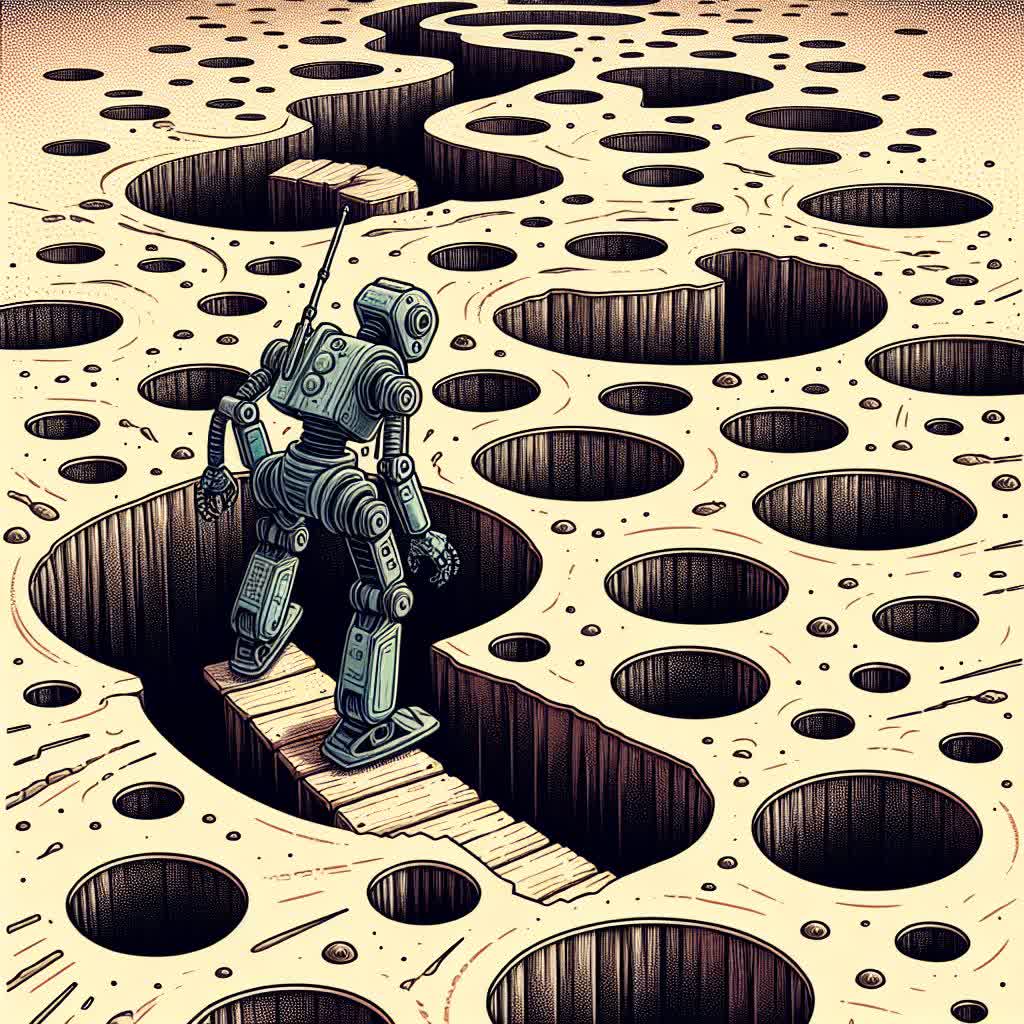}
\caption{Robot walking on the potholed path}
\label{Robot}
\end{center}
\end{figure}

Farber proved that a continuous motion planning $s: X \times X \to PX$ exists if and only if the configuration space $X$ is contractible \cite{1}. Now, we prove that a discrete motion planning exists for discrete contractible spaces.
\begin{theorem}\label{Th1}
An $r$-motion planning in a metric space $X$ exists if and only if $X$ is an $r$-contractible space.
\end{theorem}
\begin{proof}
Let $s:X\times X \to X^{[m]}$ be an $r$-motion planning in metric space $(X,d_X)$. Hence for $(x,y) \in X\times X$, $s(x,y)$ is an $r$-path from $x$ to $y$. Fix $a\in X$, define $F:X\times [m] \to X$ such that $F(x,i)=s(a,x)(i)$. Therefore, $F(x,0)=s(a,x)(0)=a$ and $F(x,m)=s_{(a,x)}(m)=x$. Let  $i\in [m]$. Since $s$ is a $1$-Lipschitz map, for every $x, x^\prime \in X$,
\begin{align*}
d_X(F(x,i), F(x^\prime , i))&=d_X(s{(a,x)}(i) , s{(a,x^\prime )}(i))\\
& \leq Max \{d_X(s{(a,x)}(j),s{(a,x^\prime )}(j))|j\in [m]\}\\&=d_{X^{[m]}}(s(a,x),s(a,x^\prime )) \\
& \leq d_{X\times X}((a,x),(a,x^\prime) )\\
&=d_X(x,x^\prime )
\end{align*}
Now let $x \in X$. Since $s(a,x)$ is an $r$-path, for every $i\in [m]$,
 $$d_X(F(x,i) , F(x,i^\prime)) = d_X(s(a,x)(i),s(a,x)(i^\prime )) \leq rd_X(i,i^\prime ).$$
 Then F is a $(1,r)$-homotopy between the identity map and the constant map $c_a$. Therefore $F:c_a\simeq_{(1,r)} id_X$.
 
Conversely, let $X$ be an $r$-contractible space. So there is a $(1,r)$-homotopy $F:X\times [m]\to X$ such that $F: id_X \simeq_{(1,r)} c_a$. We show that an $r$-motion planning in $X$ exists. For every $(x,y)\in X\times X$, let $s(x,y):[2m]\to X$ be an $r$-path from $x$ to $y$ such that 
\begin{displaymath}
s_{(x,y)}(i)= \left\{
\begin{array}{lr}
F(x,i)                     &      0 \leq i\leq m \\
F(y,2m-i)     &     m \leq i \leq 2m
\end{array}
\right.
\end{displaymath}
Now we show that $s:X\times X\to X^{[2m]}$ with $s(x,y)=s_{(x,y)}$ is a $1$-Lipschitz map. For every $(x , y) , (x^\prime , y^\prime ) \in X \times X$, we have
$$d_{X^{[2m]}}(s_{(x,y)} , s_{(x^\prime , y^\prime )}) = Max\{d(s_{(x,y)}(i),s_{(x^\prime ,y^\prime )}(i))|i\in [2m]\}$$
 We have two cases;
 \begin{itemize}
\item[(1)]
$0\leq i\leq m$, we have $d_X(s_{(x,y)}(i) , s_{(x^\prime ,y^\prime )}(i))= d_X(F(x,i), F(x^\prime , i ))\leq d_X(x,x^\prime ).$
\item[(2)]
$m\leq i \leq 2m$, we have $ d_X(s_{(x,y)}(i) , s_{(x^\prime ,y^\prime )}(i))= d_X(F(y,i), F(y^\prime , i ))\leq d_X(y,y^\prime ).$
 \end{itemize}
Thus $d_X(s_{(x,y)} , s_{(x^\prime ,y^\prime )})\leq d_{X\times X}((x,y), (x^\prime ,y^\prime ))$. Therefore $s$ is an $r$ motion planning. 
\end{proof}
By Theorem 2.8 we conclude that the space $B^n$ has an $r$-motion planning.
\begin{definition}\label{TC_r}
Let $(X,d_X)$ be a metric space. We define the discrete topological complexity of the discrete motion planning in $X$ as the minimal number $l$ such that the product $X\times X$ covered by $l$ open subsets $U_1, U_2, \ldots, U_l $ such that for every $i\in \{1,\ldots ,l\}$, there exists an $r$-motion planning $s_i : U_i \to X^{[m]}$ with $\pi \circ s_i= id_{U_i}$ for some $m \in \mathbb{N}$. In this case we set $TC_{r}(X)=l$ and if no such $l$ exists, then we set $TC_{r}(X)=\infty$.
\end{definition}
It is easy to see that if for every $1\leq i \leq l$, $s_i : U_i \to X^{[m_i]}$ is an $r$-motion planning, then for every $1\leq i \leq l$, there exists an $r$-motion planning $\overline{s}_i : U_i \to X^{[m]}$, where $m=Max\{m_1, m_2, ... m_l\}$.\\
By theorem \ref{Th1}, $TC_{r}(X)=1$ if and only if $X$ is $r$-contractible, so $TC_r(B^n)=1$.
\begin{example}
For unit circle $S^1$ we show that
\begin{displaymath}
TC_{r}(S^1)= \left\{
\begin{array}{lr}
1   &      r\geq 2 \\
2   &      r<2.
\end{array}
\right.
\end{displaymath}
For all $r\geq 2$, $S^1$ is an $r$-contractible space, then $TC_{r}(S^1)=1$. If $r<2$, then $S^1$ is not $r$-contractible, so $TC_r(S^1)\geq 2$. Now we show that  $TC_r(S^1)\leq 2$. Let $U_1=\{(a,b)\in S^1\times S^1 | a\neq b\}$ and $ U_2=\{(a,b)\in S^1\times S^1 | a\neq -b\}$. Hence $S^1\times S^1=U_1\cup U_2$ and $U_1$ ,$U_2$ are open in $S^1\times S^1$. Let $a,b\in U_1$, so $a=e^{a_1 i}$ and $b=e^{b_1 i}$. Let $s_1:U_1\to X^{[m]}$ be a map such that $m\geq \frac{2\pi}{r}$ and $s_1(a,b)= \gamma$ where
$\gamma(j)=e^{(a_1+\frac{j}{m}(b_1 - a_1))i}$.

We show that $\gamma$ is an $r$-path from $a$ to $b$. $\gamma (0)= a$ , $\gamma (m)= b$ and for every $0\leq j \leq m-1$,
 $$d(\gamma (j) , \gamma(j+1))= d(e^{(a_1+\frac{j(b_1 - a_1)}{m})i} , e^{(a_1+\frac{(j+1)(b_1 - a_1)}{m})i})$$
$$  = d(e^{(a_1+\frac{j(b_1 - a_1)}{m})i} , e^{(a_1+\frac{j(b_1 - a_1)}{m})i}\times e^{\frac{(b_1 - a_1)i}{m}}) \leq \frac{b-a}{m} = r$$
Now we prove that $s_1$ is $1$-Lipschitz i.e. for every $(a,b),(a^\prime ,b^\prime)\in U_1$ 
$ d(s_1(a, b) ,s_1(a^\prime , b^\prime)) \leq d((a,b),(a^\prime ,b^\prime))$. Let $j\in \{0,1,\cdots, m\}$, we have
\begin{align*}
d(s_1(a,b)(j),s_1(a^\prime , b^\prime)(j))
&=d(e^{(a_1+\frac{j(b_1-a_1)}{m})i},e^{(a_1^\prime + \frac{j(b_1^\prime - a_1^\prime)}{m})i})\\
&=d(e^{(\frac{jb_1}{m}+\frac{(m-j)a_1}{m})i},e^{(\frac{jb_1^\prime}{m}+\frac{(m-j)a_1^\prime}{m})i})\\
&\leq d(e^{(\frac{jb_1}{m}+(\frac{m-j}{m})a_1)i},e^{(\frac{jb_1}{m}+(\frac{m-j}{m})a_1^\prime)i})\\
&+d(e^{(\frac{jb_1}{m}+(\frac{m-j}{m})a_1^\prime)i},e^{(\frac{jb_1^\prime}{m}+(\frac{m-j}{m})a_1^\prime)i})\\
&= d(e^{\frac{m-j}{m}a_1i}, e^{\frac{m-j}{m}a_1^\prime i})+d(e^{\frac{j}{m}b_1i},e^{\frac{j}{m}b_1^\prime i}) \\
&\leq d(e^{a_1i},e^{a_1^\prime i})+d(e^{b_1i},e^{b_1^\prime i})\\
&=d(a,a^\prime)+d(b,b^\prime)\\
&=d((a,b),(a^\prime , b^\prime))
\end{align*}
An $r$-motion planning over $U_2$ is given by the map $s_2:U_2\to X^{[m]}$ which for every $(a, b)\in U_2$, $s_2(a, b)$ is the shortest counterclockwise $r$-path between $a$ and $b$.
\end{example}
\begin{example}
Let $X=S^1\vee S^1$ and $r<2$, then we show that $TC_r(X) = 2$. For this it is enough to show that $TC_r(S^1\vee S^1)\leq 2$. Let $V_1\in (S^1\vee S^1) \times (S^1\vee S^1)$ such that $V$ be union of $U_1$ in the last example for $S_1$ and $U^\prime _1$ from the last example for another $S^1$ except two element $[(1,0),(-1,0)]$ and $[(-1,0),(1,0)]$. The discrete motion planning $\sigma_1 : V_1\to  (S^1\vee S^1)^{[m]} $ be the last discrete motion plannin $s_1:U_1\to X^{[m]}$  if both element were in the same circle otherwise first, we apply the map $s$ on the first point and the connecting point of the two circles, and then we apply the map $s$ on the connecting point of the two circles and the second point, and consider the connection of these two paths. We can make $V_2$  based on $U_2$ in the same way as before. Thus $TC_r(S^1\vee S^1)= 2$. By a similar way it can be shown that if $X=S^1\vee \cdots \vee S^1$, $n$ times, then $TC_r(X)=2$.

\end{example}
\begin{example}
Let $\mathbb{HE}$ be the Hawaiian earring space \ref{Hawaiian earring} i.e. $\mathbb{HE}=\bigcup_{n\in \mathbb{N}} \{(x,y)\in \mathbb{R}^2|(x-\frac{1}{n})^2+y^2=\frac{1}{n^2}\}$. If $r<2$, then  $TC_r(\mathbb{HE})=2$. Similar to the previous example we show that $TC_r(\mathbb{HE})\leq 2$. We make $U_1,U_2 \in \mathbb{HE}\times\mathbb{HE}$ with the help of the previous example with the condition that the circles are infinite. We also consider two motion planning $s_1,s_2$ ,which are similar to the previous motion planinngs.
 Thus for $r<2$ we have $TC_r(\mathbb{HE})=2$. 
\end{example}
Also Koch snowflake and Sierpiński carpet have finite discrete topological complexity like Hawaiian earring.
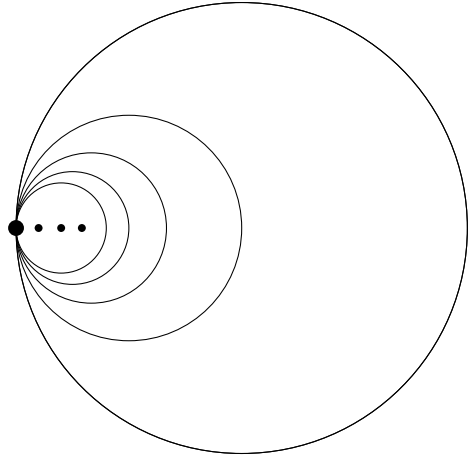
\begin{figure}
\begin{center}

\begin{tikzpicture}[scale=3]
    \draw (1,0) circle (1);
     Draw the remaining circles
    \foreach \n in {1,...,5} {
        \draw ({1/(\n)},0) circle ({1/(\n)});
    }
    \fill (0,0) circle (1pt);
    \fill (7/24,0) circle (0.5pt);
    \fill (2/10,0) circle (0.5pt);
    \fill (2/20,0) circle (0.5pt);
\end{tikzpicture}
\caption{Hawaiian earring}
    \label{Hawaiian earring}

\end{center}
\end{figure}

In the discrete case we can compare the topological complexity for different scalars.
\begin{corollary}
If $r^\prime > r > 0$, then $TC_r(X)\geq TC_{r^\prime }(X)$.
\end{corollary}

\begin{proof}
Let $s :U \to X^{[k]}$ be an $r$-motion planing. We know every $r$-path in $X$ is an $r^\prime$-path in $X$. Therefore, $s : U  \to X^{[k]}$ is an $r$-motion planning thus $TC_r(X)\geq TC_{r^\prime }(X)$.
\end{proof}

\section{Discrete Homotopy Invariance}\label{sec2}
 To define homotopy invariance in the discrete state, we need to introduce two spaces that have the same discrete homotopy type.
\begin{definition}
Two space $X$ and $Y$ have the same $r$-homotopy type if there are $r_1$-Lipschitz map $f:X\to Y$ and $r_2$-Lipschitz map $g:Y\to X$ such that $r_1r_2\leq 1$ and $f\circ g \simeq_{(1,r)} id_Y$ , $g\circ f\simeq_{(1,r)} id_X$. Moreover, if $r_1=1=r_2$, then we call $X$ and $Y$ have the same strictly $r$-homotopy type.
\end{definition}

\begin{proposition}
A space $X$ has the same $r$-homotopy type as a point if and only if $X$ is $r$-contractible. 
\end{proposition}
\begin{proof}
Let $X$ and one point space $\{a\}$ have the same $r$-homotopy type. There are $r_1$-Lipschitz map  $f:X \to \{a\}$, and $r_2$-Lipschitz map $g:\{a\} \to X$ with $g(a)=x_0\in X$  such that $r_1 r_2\leq 1$ and $f\circ g \simeq_{(1,r)} id_{\{a\}}$ and $g\circ f\simeq_{(1,r)} id_X$. For all $x\in X$, $g\circ f(x)=g(a)=x_0$ so $g\circ f$ is constant. Hence $id_X$ and the constant map $g\circ f$ are $r$-homotopic.
Conversely, let X be an $r$-contractible space. Then there is a $(1,r)$-homotopy $F:X\times [m]\to X$ between the identity map $id_X$ and the constant map $c_{x_0}$. Define $f:X\to \{x_0\}$ such that $f(x)=x_0$ and $g: \{x_0\} \to X$ such that $g(x_0)= x_0$. We have $f,g$ are $1$-Lipschitz. Also, $f\circ g(x_0)=id_{\{x_0\}}$  so $f\circ g\simeq_{(1,r)} id_{\{x_0\}}$ and for $x\in X$,  $g\circ f(x)= g(x_0) \simeq_{(1,r)} id_X$. Thus $X$ and one point space $\{x_0\}$ have the same $r$-homotopy type.
\end{proof}
\begin{theorem}
Let $f:X\to Y$ be an $r_1$-Lipschitz map and $g:Y\to X$ be an $r_2$-Lipschitz map such that $r_1r_2\leq 1$ , $ f\circ g  \simeq_{(1,r)} id_Y$ and $g\circ f \simeq_{(1,r)} id_X$. Then $TC_r(Y)\leq TC_{\frac{r}{r_1}}(X)$ and $TC_r(X)\leq T_{\frac{r}{r_2}}(Y)$.
\end{theorem}
\begin{proof}
Let $X$ and $Y$ have the same $r$-homotopy type, so  there are $r_1$-Lipschitz map $f:X\to Y$ and $r_2$-Lipschitz map $g:Y\to X$ such that $r_1r_2\leq 1$ and $F:Y\times [m]\to Y$ be a $(1,r)$-homotopy between $f\circ g$ and $id_Y$ and $G$ be a $(1,r)$-homotopy between $g\circ f$ and $id_X$. We show that $TC_r(Y)\leq TC_{\frac{r}{r_1}}(X)$. Let $TC_{\frac{r}{r_1}}(X)= l$, so there is an open cover $\{U_i\}_1^l$ for $X$, such that for every $0\leq i \leq l$ there exists an ${\frac{r}{r_1}}$-motion planning $s_i: U_i \to X^{[k]}$ for $U_i$. Suppose that $V_i=(g\times g)^{-1}(U_i)$. Since $g$ is an $r_2$-Lipschitz map, so $g\times g$ is continuous, so   $V_i$ is open in $Y\times Y$ and $\{V_i\}_1^l$ is an open cover for $Y\times Y$. Let $1\leq i\leq l$, we define an $r$-motion planning $\sigma_i:V_i \to Y^{[2m+k]}$, where  $\sigma_i(y,z)= \gamma_y^{-1}\ast \gamma_{yz} \ast \gamma_z $ such that for every $y\in Y$, $\gamma_y(j)=F(y,j)$, $0 \leq j\leq m$ and for every $(y,z)\in V_i$, $\gamma_{yz} (j) = f(s_i(g(y),g(z))(j))$, $0 \leq j\leq k$. It's clear to see that $\sigma_i(y,z)$ is an $r$-path from $y$ to $z$. We show that $\sigma_i $ is a $1$-Lipschitz map. Note that
$$d(\sigma_i (y,z),\sigma_i (y^\prime , z^\prime ) ) = d(\gamma_y^{-1}\ast \gamma_{yz} \ast \gamma_z , \gamma_{y^\prime}^{-1} \ast \gamma_{y^\prime z^\prime} \ast \gamma _{z^\prime})$$
$$=Max \{ d_Y((\gamma_y^{-1} \ast \gamma_{yz} \ast \gamma_z )(j) , (\gamma_{y^\prime}^{-1} \ast \gamma_{y^\prime z^\prime} \ast \gamma _{z^\prime}) (j))| j\in [2m+k] \}$$

we have three cases;

\begin{itemize}
\item[(1)] For every $0\leq j\leq m$, $d_Y(\gamma _y^{-1}(j) , \gamma _{y^\prime}^{-1}(j))=d(F(y,m-j),F(y^\prime , m-j))\leq d(y,y^\prime)\leq d((y,z), (y^\prime,z^\prime))$
\item[(2)] For every $0 \leq j \leq k$,
\begin{align*}
d_Y(\gamma_{yz} (j) , \gamma_{y^\prime z^\prime} (j))&= d_X(f(s_i(g(y),g(z))(j)),f(s_i(g(y^\prime ) , g(z^\prime ))(j))\\
&\leq r_1d(s_i(g(y),g(z))(j),s_i(g(y^\prime ) g(z^\prime)(j) \\
&\leq r_1d(s_i(g(y),g(z)),s_i(g(y^\prime),g(z^\prime)))\\
&\leq r_1 d((g(y),g(z)),(g(y^\prime),g(z^\prime)))\\
&= r_1((d(g(y),g(y^\prime))+d(g(z),g(z^\prime)))\\
&\leq r_1r_2(d(y,y^\prime)+d(z,z^\prime))\\
&\leq d((y,z), (y^\prime,z^\prime))
\end{align*}

\item[(3)] For every $0 \leq j \leq k$, we have $d_Y(\gamma _y(j) , \gamma _{y^\prime} (j)=d(F(z,j),F(z^\prime , j))\leq d(z,z^\prime)\leq d((y,z), (y^\prime,z^\prime))$
\end{itemize}
Then $TC_r(Y)\leq TC_{\frac{r}{r_1}}(X)$. Similarly, we can show that $TC_r(X)\leq TC_{\frac{r}{r_2}}(Y)$.	
\end{proof}

From the Theorem 3.3 we have the following result.

\begin{corollary}
If $X$ and $Y$ have the same strictly $r$-homotopy type and $r>0$, then $TC_r(X)\ = TC_r(Y)$.
\end{corollary}

\section{ Upper and lower Bound for $TC_r(X)$} \label{sec3}
Recall from \cite{3} that $cat(X)$ is defined as the smallest integer $k$ such that $X$ may be covered by $k$ open subsets $V_1 \cup\cdots\cup V_k = X$ with each inclusion $V_i \hookrightarrow X$ null-homotopic (see also \cite{7, 8}). Now we define discrete category for metric space. we need to define discrete categorical subset.
\begin{definition}
A subset $A$ from metric space $X$ is an $r$-categorical if the inclusion map $i: A \to X$ is an $r$-null homotopy.
\end{definition}

\begin{example}
Let $r>0$ and $n>1$. There is an integer $m$ such that $\dfrac{1}{m}<r$. Define a map $F:‌S^{n-1} \times [m] \to D^n$ such that $F(x,i)=x-\dfrac{ix}{m}$. Hence, $F$ is a $(1,r)$-homotopy between the inclusion map $i: ‌S^{n-1} \to D^n$ and the zero map. Therefore, the inclusion map $i$ is an $r$-null homotopy, which implies that $‌S^{n-1}$ is an $r$-categorical in $D^n$. Moreover, if we remove some sufficiently small balls from $D^n$, then $‌S^{n-1}$ is an $r$-categorical in the remaining space.
\end{example}
\begin{center}
\begin{tikzpicture}
    \def\n{5}
    \def\radius{1.5}
    \foreach \i in {1,...,\n} {
        \draw (0,0) circle (\i*\radius/\n);
    }
    \foreach \i in {1,...,\n} {
        \pgfmathsetmacro\label{\n+1-\i}
        \node[below] at (0.2,-\i*\radius/\n) {};
    }
    \draw[->,thick] (1,1) -- (0.2,0.2);
    \node[above right]at (0.9,0.9) {$F(x,i)$};
    \draw[->] (-2,0) -- (2,0) node[right] {};
        \draw[->] (0,-1.8) -- (0,1.8) node[above]{} ;
   \draw[fill=black] (0.45,0) circle (0.1);
      \draw[fill=black] (0.75,0.2) circle (0.1);
         \draw[fill=black] (-1,-0.9) circle (0.1);
       \draw[fill=black] (-1,0.9) circle (0.1);
          \draw[fill=black] (1,-0.9) circle (0.1);
          \draw[fill=black] (-0.75,-0.2) circle (0.1);
   \fill (0,0) circle (1pt);
   \
\end{tikzpicture}
\end{center}
\begin{definition}
The $r$-category of metric spase $X$ is the smallest integer $k$ such that An $r$-categirical open cover $\{v_i\}_{i=1}^k$ on $X$ exists. We denote $cat_r(X)=k$ and otherwise $cat_r(X)=\infty$.
\end{definition}
\begin{corollary}
If $X$ is an $r$-contractible space, then $cat_r(X)=1$.
\end{corollary}

Now we explain relation between $TC_r(X)$ and $cat_r(X)$.
\begin{theorem}
If $X$ is a metric space, then $cat_r(X)\leq TC_r(X)$. Moreover $TC_r(X) \leq cat_r(X\times X)$ if $X$ is $r$-connected.
\end{theorem}
\begin{proof}
Let $TC_r(X)=l$. There are open subsets $U_1 , \cdots , U_l$ in $X\times X$ such that $X\times X= U_1\cup \cdots \cup U_l$ and  for every $1\leq i\leq l$ there is an $r$-motion planning $s_i:U_i\to X^{[m]}$. Let $x_0\in X$. For $i\in \{1,\cdots ,l\}$, since $U_i$ is open in $X\times X$, the subset $(\{x_0\}\times X)\cap U_i$ is open in $\{x_0\}\times X$. So there is an open set $V_i$ in $X$ such that $ \{x_0\}\times V_i= U_i\cap (\{x_0\}\times X)$. where $V_i=\{x\in X |(x_0,x)\in U_i\}$. Since $ s_i|_{U_i\cap (\{x_0\}\times X)}$ is an $r$-motion planning for $U_i\cap ( \{x_0\}\times X)$, and  $U_i\cap (\{x_0\}\times X)=\{x_0\}\times V_i$, consider $F_i:V_i \times [m]\to X$ by $F_i(x,j)=s_i|_{\{x_0\}\times V_i}(x_0,x)(j)$. So $F(x,0)=s_i|_{\{x_0\}\times V_i}(x_0,x)(0)=x_0$ and  $F_i(x,m)=s_i|_{\{x_0\}\times V_i}(x_0,x)(m)=x$. Now we show that $F$ is a $(1,r)$ homotopy between the inclusion map and a constant map. For $j\in [m]$ , let $x_1 , x_2\in X$. Since $s_i$ is $1$-Lipschitz, \begin{align*}
d(F_i(x_1,j),F_i(x_2,j))&=d(s_i(x_0,x_1)(j),s_i(x_0,x_2)(j))\\
&\leq d(s_i(x_0,x_1),s_i(x_0,x_2))\\
&\leq d((x_0,x_1), (x_0,x_2))\\&=d(x_1,x_2)
\end{align*}
Also for every $x\in X$, let $j,j^\prime \in [m]$. Since $s_i(x_0,x)$ is an $r$-path,
$$ d(F_i(x,j),F_i(x,j^\prime ))= d(s_i(x_0,x)(j),s_i(x_0,x)(j^\prime))\leq rd(j,j^\prime )$$
Since $\{U_i\}_{i=1}^l$ is a cover for $X\times X$, we have
$$\{x_0\}\times \bigcup_{i=1}^lV_i =\bigcup_{i=1}^l(\{x_0\}\times V_i)= \bigcup_{i=1}^l(U_i\cap \{x_0\}\times X)=(\bigcup_{i=1}^lU_i)\cap (\{x_0\}\times X)=\{x_0\}\times X.$$
So $X=\bigcup_{i=1}^lV_i$. Therefore, $\{V_i\}_{i=1}^l$ is a $r$-categorical cover for $X$, which implise that $cat_r(X)\leq l$.

Let $X$ be $r$-connected and $cat_r(X\times X)=k$. There is an $r$-categorical open cover $\{V_i\}_{i=1}^k$ for $X\times X$. Let $i \in \{ 1, 2, ..., k  \}$. Since $V_i\subseteq X\times X$ is an $r$-contractible, so there is $(x_0,x_0^\prime)\in X\times X$ and $F: V_i\times [m] \to X\times X$ such that $F$ is a $(1,r)$-homotopy between the inclusion map $V_i\hookrightarrow X\times X$ and the constant map $C_{(x_0,x_0^\prime)}$. Hence $F((x_1,x_2),0)=(x_1,x_2)$ and $F((x_1,x_2),m)=(x_0,x_0^\prime)$ for every $(x_1,x_2)\in X\times X$. Consider $F=(F_1,F_2)$ where $F_1:V_i\times [m] \to X$  is the first component of $F$ and $F_2:V_i\times [m] \to X$ is the second component of $F$. Hence $F_1((x_1,x_2),0)=x_1 , F_1((x_1,x_2),m)=x_0$ $F_2((x_1,x_2),0)=x_2 , F_2((x_1,x_2),m)=x_0^\prime $. Now define $s:V_i\to X^{[2m+k]}$ by  
\begin{displaymath}
s(x_1,x_2)(j)= \left\{
\begin{array}{lr}
F_1((x_1,x_2),j)                     &      0 \leq j\leq m \\
\gamma (j-m)                        &         m \leq j\leq m+k \\
F_2((x_1,x_2),2m+k-j)     &     m+k \leq j \leq 2m+k
\end{array}
\right.
\end{displaymath}
where $\gamma : [k]\to X$ is an $r$-path from $x_0$ to $x_0^\prime$.
We show that $s:V_i\to X^{[2m+k]}$ is an $r$-motion planning. For every $(x_1,x_2)\in V_i$ , $s(x_1,x_2)(0)=F_1((x_1,x_2),0)=x_1$ , $s(x_1,x_2)(2m+k)=F_2((x_1,x_2),0)=x_2$. Also $s(x_1,x_2)$ is an $r$-path since $F_1 ((x_1,x_2),m)=x_0=\gamma (0)$, $\gamma (k)=x_0^\prime=F_2((x_1,x_2),m)$ and $\gamma$ is an $r$-path and $d(F_t((x_1,x_2),j+1),F_t((x_1,x_2),j))\leq d(F((x_1,x_2),j+1),F((x_1,x_2),j))\leq r$ for every $t\in \{1,2\}$ , $j\in \{0,m-1\}$. Also $s$ is a $1$-Lipschitz map, since for every $(x_1,x_2),(x_1^\prime,x_2^\prime)\in V_i$, $t\in \{1,2\}$ and $j \in \{ 0, 1, ...,m \}$ we have
\begin{align*}
d(F_t((x_1,x_2),j),F_t((x_1^\prime,x_2^\prime),j))&\leq d(F((x_1,x_2),j),F((x_1^\prime,x_2^\prime),j))\\&\leq d((x_1,x_2),(x_1^\prime,x_2^\prime)).
\end{align*}
\end{proof}













\end{document}